\documentclass[12pt]{amsart}

\usepackage{amssymb}
\usepackage{graphicx}
\usepackage{color}
\usepackage{hyperref}

\newtheorem{theorem}{Theorem}[section]
\newtheorem{lemma}[theorem]{Lemma}

\newtheorem{proposition}[theorem]{Proposition}

\theoremstyle{remark}

\numberwithin{equation}{section}

\newcommand{\norm}[1]{\lVert #1 \rVert}
\newcommand{\abs}[1]{\lvert #1 \rvert}
\DeclareMathOperator{\esp}{e}

\newcommand{\settc}[2]{\bigl\{\,#1 \bigm\vert #2\,\bigr\}}

\title{Traveling wave solutions to the Allen-Cahn equation}

\author{Chao-Nien Chen} 

\address[Chen]{Department of Mathematics, National Tsing Hua
University, Taiwan}

\email[Chen]{chen@math.nthu.edu.tw}

\author{Vittorio Coti Zelati}

\address[Coti Zelati]{Dipartimento di Matematica e Applicazioni ``R.
Caccioppoli'', Universit\`a degli Studi di Napoli ``Federico II'', via
Cintia, Monte S. Angelo, IT-80126 Napoli, Italy}

\email[Coti Zelati]{zelati@unina.it}



\thanks{\textbf{Acknowledgments}: Part of the work was done when Chen
was visiting at Università di Napoli Federico II and Coti Zelati was
visiting the National Tsing Hua University, Taiwan. Both Authors
thank the hosting institutions} 

\thanks{Research is supported in part by MOST 105-2115-M-007-009-MY3,
108-2115-M-007-009, Taiwan and in part by MIUR grant PRIN 2015
2015KB9WPT, ``Variational methods, with applications to problems in
mathematical physics and geometry''.}

\begin{document}

\begin{abstract}
	For the Allen-Cahn equation, it is well known that there is a
	monotone standing wave joining with the balanced wells of the
	potential. In this paper we study the existence of traveling wave
	solutions for the Allen-Cahn equation on an infinite channel. Such
	traveling wave solutions possess a large number of oscillation and
	they are obtained with the aid of variational arguments.
\end{abstract}

\maketitle

\section{Introduction}

Let $(\xi,y) \in \Omega := \mathbb{R}^1 \times \Omega_y$, a cylinder
with cross section $\Omega_y$. Here $\Omega_y$ is a bounded open set
in $\mathbb{R}^{N-1}$ with $C^{2,\alpha}$ boundary $\partial\Omega$,
$\alpha \in(0,1)$. We are concerned with the traveling wave solutions of
\begin{equation}
	\label{ac}
	\begin{cases}
		u_t=u_{\xi\xi}+\Delta_y u+u(1-u^2), \\
		u\vert_{\partial\Omega} = 0.
	\end{cases}
\end{equation}
The Allen-Cahn equation and related problems have been investigated
for several decades, because not only this nonlinear PDE serves as a
model \cite{Allen_Cahn_1979,BTW,S} in studying phase transition theory
but it has attracted a great attention from different fields in
mathematics (see e.g. \cite{dKW,M,RS} and the references therein).

Traveling wave played an important role in understanding dynamics of
evolution systems. There are many interesting results
\cite{Aronson_Weinberger_1975, Berestycki_Nirenberg_1992,
Fife_Mcleod_1977, Gallay_Risler_2007, Gardner_1986, Heinze_2001,
Lucia_Muratov_Novaga_2004, Lucia_Muratov_Novaga_2008, Reineck_1988,
Rinzel_Terman_1982, Risler_2008, Vega_1993, Vega_1993-1} for the
traveling wave solutions of
\begin{equation}
	\label{seq}
	u_t=u_{\xi\xi}+\Delta_y u+g(u), \qquad (\xi,y) \in \Omega, \ t>0.
\end{equation}
Particular examples include $g(u)=u(1-u)$ in the KPP equation and
$g(u)=u(1-u)(u-\beta)$ with $\beta \in(0,1/2)$ in the Nagumo equation.

A simple traveling wave can be represented by a function $u(\xi-ct)$
which satisfies \eqref{seq} with $c$ being a constant. Along the
moving coordinates with speed $c$, this planar traveling wave is a
solution of {an ordinary differential equation}. Such wave solutions
have been successfully investigated by shooting method
\cite{Aronson_Weinberger_1975, Fife_Mcleod_1977}. A wave with zero
speed is referred to as a standing wave. It is known that \eqref{ac}
possess a planar standing wave joining $1$ and $-1$, the double wells
of the potential. Our aim in this paper is to investigate the
traveling wave solutions of \eqref{ac}.

Following the ansatz proposed in \cite{Heinze_2001}, if $u(c(\xi-ct),
y)$ satisfies \eqref{ac} then
\begin{equation}
	\begin{cases}	
		\label{te} 
		c^2(u_{xx}+u_x)+\Delta_y u+u(1-u^2)=0, \\
		u\vert_{\partial\Omega}=0,
	\end{cases}
\end{equation}
where $x=c(\xi-ct)$. Denote by $L^p_w$ the Banach space of functions
in $L^{p}_{\text{loc}}(\Omega)$ equipped with the norm
\begin{equation*}
	\norm{u}_{L^p_w}^{p} = \int_{\Omega} e^x \abs{u}^p \, dx\, dy.
\end{equation*}
The appearence of weight function $e^x$ is due to the first order
derivative term $u_x$ in \eqref{te}. Choosing the ansatz $u(c(\xi-ct),
y)$ instead of $u(\xi-ct, y)$ seems to be more convenient in dealing
with $\Phi_c$ on function spaces with a fixed weight; for instance, in
studying the continuous dependence on $c$. Given $u \in
H^{1}_{\text{loc}}(\Omega)$ let $\norm{u}_{\textbf{E}}^2 =
\norm{u}_{L^2_w}^2 + \norm{u}_{L^4_w}^2 + \norm{u_x}_{L^2_w}^2 +
\norm{\nabla_y u}_{L^2_w}^2$. The set of functions with
$\norm{u}_{\textbf{E}} < \infty$ is denote by $\textbf{E}$, while
${\textbf{E}}_0$ is the completion of $C_0^{\infty}(\Omega)$ under the
norm $\norm{\cdot}_{\textbf{E}}$.

Let $F(w) := -\frac{w^2}{2} + \frac{w^4}{4}$ and
\begin{equation}
	\label{Phi}
	\Phi_c[w] := \frac{c^2}{2}\int_{\Omega}e^xw_x^2\, dx\, dy +
	\int_{\Omega} e^x \left(\frac{1}{2} \abs{\nabla_y w}^2 + F(w)
	\right) dx\, dy.
\end{equation}
By standard theory of calculus of variations, a critical point of
$\Phi_c$ in ${\textbf{E}}_0$ is a solution of \eqref{te}. 

Consider a cross section $\Omega_{y}$ and the boundary value problem
\begin{equation}
	\label{b}
	\begin{cases}
		\Delta_y u + u(1 - u^2) = 0,\\
		u\vert_{\partial\Omega_y}=0.	
	\end{cases}
\end{equation}
The existence of multiple solutions to \eqref{b} has been established
by Variational methods. These solutions are the critical points of the
functional $J \colon H^1_0(\Omega_y) \to \mathbb{R}$ defined by
\begin{equation}
	\label{een}
	J[v] := \int_{\Omega_y} \left(\frac{1}{2} \abs{\nabla_y v}^2 +
	F(v) \right) dy.
\end{equation}
Denote by $\lambda_{1} < \lambda_{2} \leq \lambda_{3} \leq \ldots$ the
eigenvalues of
\begin{equation} 
	\label{ev}
	\begin{cases}
		\Delta_y \psi + \lambda \psi = 0, & \\
		\psi \vert_{\partial\Omega_y} = 0. &
	\end{cases}
\end{equation}
Clearly $u \equiv 0$ is a trivial solution of \eqref{b} and $J[0] =
0$. If $\lambda_1 < 1$, it is known \cite{Chen_1990, Heinz_1986-1,
Rabinowitz_1971} that there is a unique positive solution $u_+$ for
\eqref{b} and $J[u_+]= \inf_{v \in H^1_0(\Omega_y)} J[v] < 0$ (see
proposition \ref{y}).

Our goal is to show the following result. 
\begin{theorem}
	\label{main}
	Let $\Omega_{y}$ be a $C^{2,\alpha}$ bounded domain and
	$\lambda_{j}$ be the eigenvalues of \eqref{ev}. 
	\begin{enumerate}
		
		\item [(i)] Assume that $\lambda_{1} < 1 \leq \lambda_{2}$.
		Then for every $c \in (0,2\sqrt{(1 - \lambda_{1})})$, there is
		a traveling wave solution $u$ of \eqref{ac} with wave speed
		$c$ such that $u(c(\xi-ct),y) \to 0$ as $t \to -\infty$ and
		$u(c(\xi-ct),y) \to u_+(y)$ as $t \to +\infty$.
		
		\item [(ii)] If $J$ has finite number of critical points in
		$H^1_0(\Omega_y)$, then for every $c \in (0,2\sqrt{(1 -
		\lambda_{1})})$, there is a traveling wave solution $u$ of
		\eqref{ac} with wave speed $c$ such that $u(c(\xi-ct),y) \to
		u^*$ as $t \to -\infty$ and $u(c(\xi-ct),y) \to u_*(y)$ as $t
		\to +\infty$, where $u_*$,$u^*$ are two critical points of $J$
		such that $J[u_*] < J[u^*]$.
        
		\item [(iii)] $-u$ is also a traveling wave solution of
		\eqref{ac}.
	\end{enumerate}
\end{theorem}

In particular $J$ has finite number of critical points in
$H^1_0(\Omega_y)$ if $\Omega_y$ is an interval. Indeed
\cite{Chen_1990-1,Chen_1993} there are $2k+1$ critical points if
$\lambda_{k} < 1 \leq \lambda_{k+1}$.

For the scalar reaction-diffusion equation \eqref{seq}, the ordered
method has been developed to show the existence of traveling waves on
the cylindrical domain \cite{Berestycki_Nirenberg_1992,Vega_1993-1}.
As a consequence of the maximum principle
\cite{Gilbarg_Trudinger_1983}, such traveling front solutions possess
certain monotonicity properties. For instance, let $v_+(y)$ and
$v_-(y)$ be the stable solutions of
\begin{equation}
	\label{g}
	\begin{cases}
		v_t = \Delta_y v + g(v),\\
		v\vert_{\partial\Omega_y}=0.	
	\end{cases}
\end{equation}
If
\begin{equation*}
	 v_+(y) > v_-(y) \text{ for all } y\in \Omega_y,
\end{equation*} 
Vega \cite{Vega_1993-1} proved for \eqref{seq} the existence of a
traveling front solution $w(\xi-ct,y)$ which satisfies
\begin{equation}
		v_+(y) > w(x,y)>v_-(y) \text{ for all } (x,y) \in
		\mathbb{R}\times\Omega_ y. \label{0.15}
\end{equation}
Moreover the method of moving planes and the sliding method
\cite{Berestycki_Nirenberg_1991} show that such a wave is strictly
monotone in the $x$-direction. Based on the extension of comparison
technique, the ordered method has been generalized to studying
traveling front solutions in monotone systems
\cite{Volpert_Volpert_Volpert_1994}. In an earlier work
\cite{Gardner_1986}, Gardner considered a discretization of
\eqref{seq} and applied the Conley index to establish an existence
result similar to \cite{Vega_1993-1}. The monotonicity properties for
traveling waves in combustion models have been studied in \cite{BLL}.

More recently variational methods have been employed to investigate
the traveling wave solutions for reaction-diffusion equations with
Ginzburg-Landau or bistable type nonlinearities. In \cite{Heinze_2001,
Lucia_Muratov_Novaga_2008, Lucia_Muratov_Novaga_2004}, the authors
proved existence of traveling waves via constrained minimization in a
weighted Sobolev space like ${\textbf{E}}_0$. This constrained
minimization requires the traveling front solution stay in the
weighted Sobolev space and leads the solution to have certain
monotonicity properties.

The remainder of this paper is organized as follows. Section
\ref{sec_pre} begins with some known results as a preliminary. For the
traveling wave solution with a given speed $c$, a variational
formulation is introduced in Section \ref{sec_variation} to establish
a sequence of approximated solutions through a mini-max scheme based
on the Krasnoselski genus. As the number of genus is increasing to
infinity along this sequence of approximated solutions, it is expected
\cite{Heinz_1986,Heinz_1986-1,Heinz_1987} that such a traveling wave
solution should possess a large number of oscillation as a limit of
the approximated solutions.

In using variational approach to study traveling wave solution, a
commonly used weighted Sobolev space \cite{Heinze_2001,
Lucia_Muratov_Novaga_2008, Lucia_Muratov_Novaga_2004} is the Hilbert
space $\textbf{H}$ or ${\textbf{H}}_0$ equipped with the norm
$\norm{\cdot}_{\textbf{H}}$, where $\norm{u}_{\textbf{H}}^2 =
\norm{u}_{L^2_w}^2 + \norm{u_x}_{L^2_w}^2 + \norm{\nabla_y
u}_{L^2_w}^2$. Choosing the space ${\textbf{E}}_0$ enables us to work
out the boundedness of the solutions and the compactness of
Palais-Smale sequences in dealing with the mini-max argument. Then in
Section \ref{sec_lim}, utilizing a suitable limit procedure, we
establish the traveling wave solutions. Moreover, as stated in
Theorem~\ref{main}, there are infinite number of traveling wave
solutions which are distinguished by their speed. To the best of our
knowledge, using mini-max method to establish traveling wave solution
seems to be new and such a class of traveling front solutions have not
been studied before. It should be interesting to point out that all
the traveling front solutions stated in Theorem~\ref{main} do not
belong to ${\textbf{E}}_0$ or ${\textbf{H}}_0$. Because of this fact,
we need a delicate procedure in passing to the limit form the
approximated solutions; however this procedure does not keep tracking
the number of oscillation. It is not clear to us if a direct argument
is available for the proof. We remark that the Allen-Cahn equation has
infinite number of planar traveling front solutions which satisfy
$\lim_{t \to -\infty } u(\xi-ct) = 0$ and $\lim_{t \to \infty }
u(\xi-ct) = 1$. These solutions change sign infinitely many times.

\section{Preliminary} 
\label{sec_pre}

We state some useful inequalities whose proofs can be found in
\cite{Lucia_Muratov_Novaga_2008, Muratov_2004}.

\begin{lemma}
	\label{pro1}
	If $w(x,y)$ is such that $\norm{w}_{L^2_w}^2 +
	\norm{w_x}_{L^2_w}^2 + \norm{\nabla_y w}_{L^2_w}^2 < + \infty$,
	then
	\begin{align}
		\int^{+\infty}_{r}\int_{\Omega_y} e^x w^2 \, dxdy \leq 4
		\int^{+\infty}_{r}\int_{\Omega_y} e^x w_x^2 \, dxdy,\label{h1}
		\\
		\int_{\Omega_y} w^2(r,y) \, dy \leq e^{-r} \int^{+\infty}_{r}
		\int_{\Omega_y} e^x w_x^2 \, dxdy, \label{h2}
	\end{align}
	for any $r \in \mathbb{R}$; in particular
	\begin{equation}
		\int_{\Omega} e^x w^2 \, dxdy \leq 4 \int_{\Omega} e^x w_x^2\,
		dxdy.\label{pineq}
	\end{equation}
\end{lemma}

For the non-trivial solutions of \eqref{b}, some existence and
uniqueness results can be found in \cite{Chen_1990-1, Chen_1993,
Heinz_1986, Heinz_1987, Rabinowitz_1971}. Part of such results are
stated in the next proposition.
\begin{proposition}
	\label{y} 
	Let $\Omega_{y}$ be a bounded open set in $\mathbb{R}^{n}$  
	and $\lambda_{i}$ the eigenvalues of \eqref{ev}.

	\begin{enumerate}
		\item[(a)] If $\lambda_1<1$ there is a unique positive
		solution $u_+$ for \eqref{b} and $J[u_+]= \inf_{v \in
		H^1_0(\Omega_y)} J[v] < 0$.
		
		\item[(b)] If $\lambda_1 < 1 \leq \lambda_2$ then $u_+$ and
		$-u_+$ are the only non-trivial solutions of \eqref{b}.
	\end{enumerate}
\end{proposition}

\section{Variational framework} 
\label{sec_variation}

In this section, a variational framework will be used to construct
approximation solutions to a traveling wave solution of \eqref{ac}.
Let $\Omega_* := (-\infty,0) \times \Omega_y$ and consider the
following boundary value problem:
\begin{equation}
	\label{LS}
	\begin{cases}
		c^2(u_{xx} + u_x) + \Delta_y u + u(1-u^2) = 0,\\
		u\vert_{\partial\Omega_*} = 0.
	\end{cases}
\end{equation}
Let $\mathbf{E}_{*}$ be the closure of $C^{\infty}_{0}(\Omega_{*})$ in
$\mathbf{E}$ and, for all $u \in \mathbf{E}_{*}$
\begin{equation}
	\label{gradE}
	I_c[u]:= \int_{-\infty}^{0} \int_{\Omega_y}
	\left(\frac{c^2}{2}u_x^2 + \frac{1}{2} \abs{\nabla_y u}^2
	-\frac{1}{2} u^{2} + \frac{1}{4} u^{4} \right)e^x \, dxdy.
\end{equation}

\begin{proposition}
	The functional $I_{c} \in C^{1}(\mathbf{E}_{*};
	\mathbb{R})$ and is bounded from below.
\end{proposition}

The proof is standard. We omit it. 

\begin{proposition}
	\label{prop:nobdry}
	Suppose that $u \in \mathbf{E}_{*}$ is a critical point of $I_c$,
	with $\norm{u}_{\infty} < \infty$. Then $u \in
	C^{2,\alpha}(\Omega_{*}) \cap C^{1,\alpha}(\bar{\Omega}_{*})$ and
	satisfies \eqref{LS}. Moreover $\norm{u}_{C^{1,\alpha}((-\infty,0]
	\times \bar \Omega_y)}$ is bounded; in particular, $u_x$ and
	$\nabla_y u$ are uniformly continuous in $\Omega$.
\end{proposition}

\begin{proof}
	A critical point $u$ satisfies
	\begin{equation*}
		0 = I'_{c}[u]\phi = \int_{-\infty}^{0} \int_{\Omega_y}
		\left(c^2 u_x \phi_{x} + \nabla_y u \cdot \nabla_{y}\phi -
		u\phi +u^{3}\phi \right)e^x \, dxdy
	\end{equation*}
	for $\phi \in \mathbf{E}_{*}^{\dag}$ ({the dual of
	$\mathbf{E}_{*}$), in particular for all $\phi \in
	C_{0}^{\infty}(\Omega_{*})$}. Since $u$ is bounded by assumption
	and $e^{x}$ is bounded on bounded subsets of $\Omega_{*}$, it
	immediately follows that $u \in H^{2}_{\text{loc}}(\Omega_{*})$.
	Then standard regularity theory \cite{Gilbarg_Trudinger_1983}
	shows that $u \in C^{2,\alpha}(\Omega_*) \cap
	C^{1,\alpha}(\bar{\Omega}_*)$, and thus it is a classical solution
	of \eqref{LS}.
\end{proof}

\begin{lemma}
	\label{bound}
	If $u \in \mathbf{E}_{*}$ and $I_c[u] \leq 0$ then
	\begin{align}
		\int_{\Omega_*}e^xu^4\, dxdy \leq
		4\int_{\Omega_*}e^x\, dxdy, \label{h6} \\
		\int_{\Omega_*}e^xu^2\, dxdy \leq
		2\int_{\Omega_*}e^x\, dxdy,\label{h3} \\
		\int_{\Omega_*}e^xu_x^2\, dxdy \leq
		\frac{2}{c^2}\int_{\Omega_*}e^x\, dxdy
		\label{h4} 
	\end{align}
	and
	\begin{align}
		\int_{\Omega_*}e^x\vert\nabla_y u\vert^2\,
		dxdy \leq 2\int_{\Omega_*}e^x\,
		dxdy.\label{h5}
	\end{align}
	In particular 
	\begin{equation}
		\label{eq:L4estimate}
		\norm{u}_{L^{4}_{w}} \leq \sqrt{2}\abs{\Omega_{y}}^{1/4}
		\qquad \text{for all } u \in \mathbf{E}_{*} \text{ such that }
		I_{c}[u] \leq 0.
	\end{equation}
\end{lemma}

\begin{proof}
	By Holder inequality
	\begin{align}
		\label{Ho}
		\int_{\Omega_*}e^xu^2\, dxdy \leq
		\left(\int_{\Omega_*}e^x\,
		dxdy\right)^{\frac{1}{2}}
		\left(\int_{\Omega_*}e^xu^4\,
		dxdy\right)^{\frac{1}{2}}.
	\end{align}
	Clearly $I_c[u] \leq 0$ implies that
	\begin{align}
		\label{<0}
		\int_{\Omega_*}\left(\frac{c^2}{2}u_x^2\,
		dxdy+ \frac{1}{2}\vert\nabla_y
		u\vert^2+\frac{1}{4}u^4\right)e^x\, dxdy \notag \\
		\leq \int_{\Omega_*}\frac{1}{2}e^xu^2\,
		dxdy.
	\end{align} 
	This together with \eqref{Ho} yields \eqref{h6}. Substituting
	\eqref{h6} into \eqref{Ho} gives \eqref{h3}. Then \eqref{h4} and
	\eqref{h5} easily follow from \eqref{<0}.
\end{proof}

\begin{lemma}
	\label{lem:Lpconvergence}
	Assume that $u_{n} \in \mathbf{E}_{*}$ is a sequence such that
	$I_{c}[u_{n}] \leq 0$ and $\norm{u_{n}}_{L^{6}_{w}(\Omega_{*})}
	\leq C$.
	
	Then there exists a subsequence $u_{n_{k}}$ which converges weakly
	in $\mathbf{E}_{*}$ and strongly in $L^{p}(\Omega_{*})$ for all $p
	\in [2, 4]$ to a function $\bar{u} \in \mathbf{E}_{*}$.
\end{lemma}

\begin{proof}
	It immediately follows from Lemma \ref{bound} that $u_{n}$ is
	bounded in $\mathbf{E}_{*}$. From the boundedness of $u_{n}$,
	there exists a subsequence $u_{n_{k}}$ which converges weakly to
	some $\bar{u} \in \mathbf{E}_{*}$ and strongly in
	$L^{\infty}([-L,0] \times \Omega_{y})$ for all $L > 0$.
	
	We next show that $u_{n_{k}} \to \bar{u}$ in
	$L^{p}_{w}(\Omega_{*})$ if $p \in [2,4]$.
	Let us first remark that Lemma \ref{bound} implies that
	\begin{equation*}
		\int_{\Omega_{*}} \esp^{x} \abs{u_{n_{k}}}^{p} \, dx dy \leq C
	\end{equation*}	
	for $2 \leq p \leq 4$, and the same inequality holds for
	$\bar{u}$, with a constant $C$ not depending on $k$.
	
	Given $\epsilon > 0$, since
	\begin{align*}
		&\int_{-\infty}^{-L} \int_{\Omega_{y}} \esp^{x} \abs{u_{n_{k}}
		- \bar{u}}^{2} \, dydx \\
		&\qquad \leq \left( \int_{-\infty}^{-L} \int_{\Omega_{y}}
		\esp^{x} \, dydx \right)^{1/2} \left(\int_{-\infty}^{-L}
		\int_{\Omega_{y}} \esp^{x} \abs{u_{n_{k}} - \bar{u}}^{4} \,
		dydx \right)^{1/2} \\
		&\qquad \leq \abs{\Omega_{y}}^{1/2} e^{-L/2}
		\left(\int_{-\infty}^{0} \int_{\Omega_{y}} \esp^{x}
		\abs{u_{n_{k}} - \bar{u}}^{4} \, dydx \right)^{1/2} \\
		&\qquad \leq \tilde{C} e^{-L/2},
	\end{align*}
	we take $L > 0$ such that 
	\begin{equation*}
		\int_{-\infty}^{-L} \int_{\Omega_{y}} \esp^{x} \abs{u_{n_{k}}
		- \bar{u}}^{2} \, dydx \leq \tilde{C} e^{-L/2} <
		\frac{\epsilon^{2}}{2}
	\end{equation*}
	and then $k_{0} \in \mathbb{N}$ such that
	\begin{equation*}
		\int_{-L}^{0} \int_{\Omega_{y}} \esp^{x} \abs{u_{n_{k}} -
		\bar{u}}^{2} \, dydx \leq \frac{\epsilon^{2}}{2} \qquad
		\text{for all } k \geq k_{0}.
	\end{equation*}
	Then for all $k \geq k_{0}$
	\begin{multline*}
		\norm{u_{n_{k}} - \bar{u}}_{L^{2}_{w}(\Omega_{*})}^{2} =
		\int_{-\infty}^{0} \int_{\Omega_{y}} \esp^{x} \abs{u_{n_{k}} -
		\bar{u}}^{2} \, dydx \\
		= \int_{-\infty}^{-L} \int_{\Omega_{y}} \esp^{x}
		\abs{u_{n_{k}} - \bar{u}}^{2} \, dydx + \int_{-L}^{0}
		\int_{\Omega_{y}} \esp^{x} \abs{u_{n_{k}} - \bar{u}}^{2} \,
		dydx < \epsilon^{2}
	\end{multline*}
	and $\norm{u_{n_{k}} - \bar{u}}_{L^{2}_{w}(\Omega_{*})}^{2} \to
	0$.
	
	Observe that
	\begin{multline*}
		\norm{u_{n_{k}} - \bar{u}}_{L^{4}_{w}(\Omega_{*})}^{4} =
		\int_{-\infty}^{0} \int_{\Omega_{y}} \esp^{x} \abs{u_{n_{k}} -
		\bar{u}}^{4} \, dydx \\
		\leq \left( \int_{\Omega_{*}} \esp^{x} \abs{u_{n_{k}} -
		\bar{u}}^{2} \, dydx \right)^{1/2} \left( \int_{\Omega_{*}}
		\esp^{x} \abs{u_{n_{k}} - \bar{u}}^{6} \, dydx \right)^{1/2} 
		\\
		\leq C \norm{u_{n_{k}} - \bar{u}}_{L^{2}_{w}(\Omega_{*})}
	\end{multline*}
	by the boundedness of $u_{n_{k}}$ in $L^{6}_{w}(\Omega_{*})$. Now
	the lemma follows.
\end{proof}

\begin{lemma}
	\label{lem:PS}
	Assume that $u_{n} \in \mathbf{E}_{*}$ is a sequence such that
	$\norm{u_{n}}_{L^{6}_{w}} \leq C$, $I_{c}[u_{n}] \to b \leq 0$ and
	$I'_{c}[u_{n}] \to 0$.

	Then there exists a subsequence $u_{n_{k}}$ such that $u_{n_{k}}
	\to \bar{u} \in \mathbf{E}_{*}$, $I_{c}[\bar{u}] = b$ and
	$I'_{c}[\bar{u}] = 0$.
\end{lemma}

\begin{proof}
	Lemma \ref{lem:Lpconvergence} implies that there is a subsequence
	$u_{n_{k}}$ which converges weakly in $\mathbf{E}_{*}$ and
	strongly in $L^{2}(\Omega_{*})$ to a function $\bar{u} \in
	\mathbf{E}_{*}$. Then
	\begin{align*}
		\int_{\Omega_{*}} &\left[ c^{2} \abs{(u_{n_{k}})_x -
		\bar{u}_{x}}^{2} + \abs{\nabla_{y} u_{n_{k}} - \nabla_{y}
		\bar{u}}^{2} \right] \esp^{x}dxdy \\
		&= \int_{\Omega_{*}} \left[ c^{2} (u_{n_{k}})_x((u_{n_{k}})_x
		- \bar{u}_{x}) + \nabla_{y} u_{n_{k}} (\nabla_{y} u_{n_{k}} -
		\nabla_{y} \bar{u}) \right] \esp^{x}dxdy \\
		&\qquad - \int_{\Omega_{*}} \left[
		c^{2} \bar{u}_{x}((u_{n_{k}})_x - \bar{u}_{x}) + \nabla_{y}
		\bar{u} (\nabla_{y} u_{n_{k}} - \nabla_{y} \bar{u}) \right]
		\esp^{x}dxdy \\
		&= \langle I'_{c}[u_{n_{k}}], u_{n_{k}} - \bar{u} \rangle +
		\int_{\Omega_{*}}  \left[ u_{n_{k}}(u_{n_{k}} - \bar{u}) -
		u^{3}_{n_{k}}(u_{n_{k}} - \bar{u}) \right] \esp^{x}dxdy \\
		&\qquad - \int_{\Omega_{*}} \left[
		c^{2} \bar{u}_{x}((u_{n_{k}})_x - \bar{u}_{x}) + \nabla_{y}
		\bar{u} (\nabla_{y} u_{n_{k}} - \nabla_{y} \bar{u}) \right]
		\esp^{x}dxdy,
	\end{align*}
	which converges to zero since $u_{n_{k}} - \bar{u}$ is bounded,
	converges weakly in $\textbf{E}_{*}$, strongly in
	$L^{2}_{w}(\Omega_{*})$ to $0$ while $u_{n_{k}}$ is bounded in
	$L^{6}_{w}(\Omega_{*})$. This immediately deduces that
	$I'_{c}[\bar{u}] = 0$ and $I_{c}[\bar{u}] = b$.
\end{proof}

\begin{lemma}
	\label{lem:gradLip}
	There exists $L_* > 0$ such that
	\begin{equation*}
		\norm{I_{c}'[u] - I_{c}'[v]}_{\mathbf{E}_{*}^{\dag}} \leq L_*
		\norm{u - v}_{\mathbf{E}_{*}}
	\end{equation*}
	if $u$, $v \in \mathbf{E}_{*}$, $I_{c}[u] \leq 0$
	and $I_{c}[v] \leq 0$.
\end{lemma}

\begin{proof}
	Suppose that $u$, $v \in \mathbf{E}_{*}$ and $I_{c}[u] \leq 0$,
	$I_{c}[v] \leq 0$. Set $h = u - v$, then for all $\phi \in
	\mathbf{E}_{*}$ 
	\begin{multline*}
		(I_{c}'[u] - I_{c}'[v])[\phi] \\
		= \int_{\Omega_{*}} \left(c^2 h_{x} \phi_{x} + \nabla_{y} h
		\cdot \nabla_{y}\phi - h \phi + (u^{3} - v^{3}) \phi \right)
		e^x \, dxdy.
	\end{multline*}
	By \eqref{eq:L4estimate}
	\begin{align*}
		&\int_{\Omega_{*}} \left( (u^{3} - v^{3}) \phi \right) e^x \,
		dxdy \\
		&\qquad = \int_{\Omega_{*}} \left( (u - v) (u^{2} + uv +
		v^{2}) \phi \right) e^x \, dxdy \\
		&\qquad \leq \norm{u - v}_{L^{4}_{w}}(\norm{u}^{2}_{L^{4}_{w}}
		+ \norm{u}_{L^{4}_{w}} \norm{v}_{L^{4}_{w}} +
		\norm{v}^{2}_{L^{4}_{w}}) \norm{\phi}_{L^{4}_{w}} \\
		&\qquad \leq C \norm{u - v}_{L^{4}_{w}}
		\norm{\phi}_{L^{4}_{w}},
	\end{align*}
	with the constant $C$ not depending on $u$ or $v$. Thus the proof
	is complete.
\end{proof}

We now introduce the min-max classes:
\begin{equation*}
	\Gamma_{k} = \settc{ A \subset \textbf{E}_{*} \setminus \{0\} }{A
	\text{ closed, } A = - A \text{ and } \gamma(A) \geq k}
\end{equation*}
where $\gamma(A)$ is the Krasnoselski genus, and
\begin{equation*}
	\hat{\Gamma}_{k} = \settc{A \in \Gamma_{k}}{\abs{u(x,y)} \leq 1,
	\text{ for all } u \in A \text{ and } (x,y) \in \Omega_{*}}.
\end{equation*}
For fixed $c$, we denote $I^{a} = \settc{u \in \mathbf{E}_{*}}{I_{c}[u] <
a}$, the sublevels of $I_{c}$. 
\begin{proposition}
	For all $k \in \mathbb{N}$, there is a set $A$ of genus $k$ such
	that
	\begin{equation}
		\label{eq:ckequalcbark}
		c_{k} = \inf_{A \in \Gamma_{k}} \sup_{u \in A} I_{c}[u] =
		\inf_{\hat{A} \in \hat{\Gamma}_{k}} \sup_{u \in \hat{A}}
		I_{c}[u] < 0
	\end{equation}
\end{proposition}

\begin{proof}
	Take $L > 0$ (to be fixed later) and consider the linear problem
	\begin{equation*}
		\begin{cases}
			-c^{2}(u_{xx} + u_{x}) - \Delta_{y}u - u = \lambda_{1} u &
			\text{in } \Omega_{L} = [-L, 0] \times \Omega_{y}, \\
			u = 0 & \text{on } \partial \Omega_{L},
		\end{cases}
	\end{equation*}
	where $\lambda_{1}$ is the first eigenvalue of \eqref{ev} and
	$\varphi_{+}$ the corresponding eigenfunction. Set
	\begin{equation*}
		\phi_{k}(x,y) = \esp^{-x/2} \sin \left( \frac{k\pi}{L} x
		\right) \varphi_{+}(y)
	\end{equation*}
	Notice that $\phi_{k}(x,y) = 0$ for $(x,y) \in \partial
	\Omega_{L}$ and it is a solution of
	\begin{equation} 
		\label{phyeq}
		-c^{2}(\phi_{xx} + \phi_{x}) - \Delta_{y} \phi - \phi =
		\left[c^{2}\left(\frac{1}{4} + \frac{k^{2}\pi^{2}}{L^{2}}
		\right) + (\lambda_{1} - 1) \right] \phi.
	\end{equation}
	Multiplying \eqref{phyeq} by $\phi \esp^{x}$ and integrating it,
	we obtain
	\begin{align*}
		\int_{-L}^{0} \int_{\Omega_{y}} &\left[ -c^{2}(\phi_{xx} +
		\phi_{x}) - \Delta_{y} \phi - \phi \right] \phi \esp^{x}dydx
		\\
		&= \left[c^{2}\left(\frac{1}{4} + \frac{k^{2}\pi^{2}}{L^{2}}
		\right) + (\lambda_{1} - 1) \right] \int_{-L}^{0}
		\int_{\Omega_{y}} \phi^{2} \esp^{x} \, dydx.
	\end{align*}
	Set
	\begin{align*}
		Q_{L}[\phi] = \frac{1}{2} \int_{-L}^{0} \int_{\Omega_{y}}
		\left[ c^{2} \phi^{2}_{x} + \abs{\nabla_{y} \phi}^{2} -
		\phi^{2} \right] \esp^{x} \, dydx.
	\end{align*}
	Since
	\begin{align*}
		2Q_{L}[\phi] = &\int_{-L}^{0} \int_{\Omega_{y}} \left[ -c^{2}
		\left( (\phi_{x}\phi\esp^{x})_x - \phi^{2}_{x}\esp^{x} \right)
		+ \abs{\nabla_{y} \phi}^{2} \esp^{x} - \phi^{2} \esp^{x}
		\right]dydx \\
		&= \int_{-L}^{0} \int_{\Omega_{y}} \left[ -c^{2}(\phi_{xx} +
		\phi_{x}) - \Delta_{y} \phi - \phi \right] \phi \esp^{x}dydx,
	\end{align*}
	it follows that $Q_{L}[\phi_{k}] < 0$ if
	\begin{equation*}
		c^{2}\left(\frac{1}{4} + \frac{k^{2}\pi^{2}}{L^{2}} \right) +
		(\lambda_{1} - 1) < 0.
	\end{equation*}
	In fact, for every given $k \in \mathbb{N}$, if $L$ is  
	large enough then $Q_{L}[\phi_{k}] < 0$, because $\lambda_{1} <
	1$ and $c < 2\sqrt{(1 -\lambda_{1})}$. 
	Clearly
	\begin{equation*}
		Q_{L}[\alpha_i \phi_{i} + \alpha_j \phi_{j}] = \alpha_i^{2}
		Q_{L}[\phi_{i}] + \alpha_j^{2} Q_{L}[\phi_{j}].
	\end{equation*}
	We now extend the functions $\phi_{i}(x,y)$ to be defined on
	$\Omega_{*}$
	by setting $\phi_{i}(x,y) = 0$ for all $x < -L$.

	Consider an odd map
	\begin{equation*}
		\xi \colon S^{k} \to \textbf{E}_{*}, \qquad \xi(\alpha_{1},
		\alpha_{2}, \ldots, \alpha_{k}) = \epsilon \sum_{i=1}^{k}
		\alpha_{i} \phi_{i}.
	\end{equation*}
	By direct calculation 
	\begin{align*}
		I_{c}[\epsilon\xi(\alpha_{1}, \ldots, \alpha_{k})] &=
		\epsilon^{2} \sum_{i = 1}^{k} \alpha_{i}^{2} Q_{L}[\phi_{i}]
		\\
		&+ \epsilon^{4} \int_{-L}^{0} \int_{\Omega_{y}}
		\frac{\abs{\sum_{i=1}^{k} \alpha_{i} \phi_{i}}^{4}}{4}
		\esp^{x} \, dx dy,
	\end{align*}
	which is negative for all $(\alpha_{1}, \ldots, \alpha_{k}) \in
	S^{k}$, provided that we pick $L$ to be large enough and
	$\epsilon$ small enough. 
	Since
	\begin{equation*}
		A = \xi(S^{k}) \in \Gamma_{k},
	\end{equation*}
	it follows that
	\begin{equation*}
		c_{k} = \inf_{A \in \Gamma_{k}} \sup_{u \in A} I_{c}[u] < 0
	\end{equation*}
	for all $k \in \mathbb{N}$.
	
	Let
	\begin{equation*}
		\hat{c}_{k} = \inf_{A \in \hat{\Gamma}_{k}} \sup_{u
		\in A} I_{c}[u].
	\end{equation*}
	Clearly $c_{k} \leq \hat{c}_{k}$, while setting a truncation
	$\hat{u}$ from $u$ by
	\begin{equation*}
		\hat{u}(x,y) =
		\begin{cases}
			u(x,y) & \text{if } \abs{u(x,y)} \leq 1 \\
			\frac{u(x,y)}{\abs{u(x,y)}} & \text{if } \abs{u(x,y)} > 1
		\end{cases}
	\end{equation*}
	yields $I_{c}[\hat{u}] \leq I_{c}[u]$ for all $u \in
	\textbf{E}_{*}$. Given any $A \in \Gamma_{k}$, define
	\begin{equation*}
		\hat{A} = \settc{\hat{u}}{u \in A}.
	\end{equation*}
	Since the map $u \mapsto \hat{u}$ is continuous in
	$\textbf{E}_{*}$, we conclude that $\hat{A} \in \hat{\Gamma}_{k}$,
	$c_{k} = \hat{c}_{k}$ and \eqref{eq:ckequalcbark} holds. 
\end{proof}

\begin{proposition}
	\label{prop:PSbounded}
	For $k \in \mathbb{N}$, let $\hat{A}_{n} \in
	\hat{\Gamma}_{k}$ be such that 
	\begin{equation*}
		c_{k} \leq \sup_{u \in \hat{A}_{n}} I_{c}[u] \leq c_{k} +
		\frac{1}{n} < 0.
	\end{equation*}
	
	Then there is a $u_{n} \in \hat{A}_{n}$ such that
	\begin{equation*}
		c_{k} - \frac{2}{n} \leq I_{c}[u_{n}] \leq c_{k} + \frac{1}{n}
		\qquad \norm{I'_{c}[u_{n}]}_{\mathbf{E}^{\dag}_{*}} \leq 8
		\sqrt{\frac{2L_*}{n}}.
	\end{equation*}
\end{proposition}

\begin{proof}
	The proof is based on the deformation theory. Suppose the
	assertion is false, then
	\begin{equation*}
		\norm{I'_c[v]}_{\mathbf{E}_{*}^{\dag}} > 8
		\sqrt{\frac{L_*}{n}}
	\end{equation*}
	for all $v \in \hat{A}_{n}$ such that $c_{k} - \frac{2}{n} \leq
	I_{c}[v] \leq c_{k} + \frac{1}{n}$. Let $\delta =
	\frac{2}{\sqrt{nL_{*}}}$ with $L_*$ given by Lemma
	\ref{lem:gradLip}. If $u \in I_{c}^{-1}([c_{k} - \frac{2}{n},
	c_{k} + \frac{2}{n}])$ and $\norm{u - v}_{\mathbf{E}_{*}} <
	2\delta$, invoking Lemma \ref{lem:gradLip} yields
	\begin{multline*}
		\norm{I'_{c}[u]}_{\mathbf{E}_{*}^{\dag}} \geq
		\norm{I'_{c}[v]}_{\mathbf{E}_{*}^{\dag}} - \norm{I'_{c}[v] -
		I'_{c}[u]}_{\mathbf{E}_{*}^{\dag}} \\
		\geq 8 \sqrt{\frac{L_*}{n}} - 2L_*\delta = 4
		\sqrt{\frac{L_*}{n}}.
	\end{multline*}
	We can then apply Lemma 2.3 of \cite{Willem_96} with $S =
	\hat{A}_{k}$ and $\epsilon = \frac{1}{n}$ since
	\begin{equation*}
		\frac{8\epsilon}{\delta} = \frac{8}{n} \frac{\sqrt{nL_*}}{2} =
		4 \sqrt{\frac{L_*}{n}}.
	\end{equation*}
	As a consequence
	\begin{equation*}
		\eta(1,\hat{A}_{n}) \subset I_{c}^{c_{k}-1/n}
	\end{equation*}
	and we have reached a contradiction since $\eta(1,\hat{A}_{n}) 
	\in \Gamma_{k}$.
\end{proof}

The following result follows from an application of the
Ljusternik-Schnirelman theory. We refer to
\cite{Heinz_1987,Rabinowitz86-1} for related applications to
differential equations.
\begin{proposition}
	\label{prop:LS}
	Let $\lambda_{1} < 1$ and $c < 2\sqrt{(1 - \lambda_{1})}$. Then
	there exist a sequence of critical points $\{\hat u_k\}$ of
	$I_{c}$ such that $I_c[\hat u_k] \leq I_c[\hat u_{k+1}] < 0$,
	$\abs{u(x,y)} \leq 1$ for all $(x,y) \in \Omega_{*}$,
	\begin{equation}
		\label{eq:Iuk}
		\lim_{k\rightarrow+\infty}I_c[\hat u_k] = 0
	\end{equation}
	and
	\begin{equation}
		\label{eq:>0}
		\int_{-\infty}^{0}\int_{\Omega_y}e^x(\hat u_k)_x^2\, dxdy>0.
	\end{equation}
\end{proposition}

\begin{proof}
	It has been shown that $\hat{c}_{k} = c_{k}$ is a critical level.
	Following from Proposition \ref{prop:PSbounded}, we can find a
	Palais-Smale sequence $v_{n}$ at level $c_{k}$ such that
	$\abs{v_{n}(x,y)} \leq 1$ for all $(x,y) \in \Omega_{*}$. By Lemma
	\ref{lem:PS} we deduce that $v_{n}$ converge to a critical point
	$\hat{u}_{k}$ at level $c_{k}$ such that $\abs{\hat{u}_{k}(x,y)}
	\leq 1$ for all $(x,y) \in \Omega_{*}$. Thus we get a sequence of
	critical points $\{\hat u_k\}$ such that $I_c[\hat u_k] = c_{k}$.

	If $\int_{-\infty}^{0}\int_{\Omega_y}e^x(\hat u_k)_x^2\, dxdy = 0$
	for some $k \in \mathbb{N}$, then $\hat u_k \equiv 0$; however
	this would violate $I_c[\hat u_k] < 0$, thus \eqref{eq:>0} must
	hold.
	
	To show \eqref{eq:Iuk}, we can follow a variant of a
	rather standard procedure (see e.g. \cite[Theorem
	10.10]{Ambrosetti_Malchiodi_2007}).
	Let
	\begin{equation*}
		\mathcal{B}_{C} = \settc{u \in
		\mathbf{E}_{*}}{\norm{u}_{L^{6}_{w}} \leq C}.
	\end{equation*}
	Notice that if $u \in \mathbf{E}_{*}$ and
	$\norm{u}_{L^{\infty}(\Omega_{*})} \leq 1$ then
	$\norm{u}_{L^{6}_{w}} \leq \abs{\Omega_{y}}^{1/6}$. Since
	$\abs{u(x,y)} \leq 1$ implies $u \in \mathcal{B}_{{C}}$ for all $C
	\geq \bar{C} = \abs{\Omega_{y}}^{1/6}$, we can find sets with
	genus $k$ in $I^{0} \cap \mathcal{B}_{C}$ for all $k \in
	\mathbb{N}$.

	Suppose that 
	\begin{equation*}
		\lim_{k \to +\infty} I_{c}[\hat{u}_{k}] = \lim_{k \to +\infty}
		c_{k} = \chi < 0.
	\end{equation*}
	Then $\gamma(I_{c}^{\chi+\epsilon} \cap \mathcal{B}_{\bar{C}}) =
	+\infty$ for all $\epsilon > 0$ such that $\chi + \epsilon < 0$.
	Since the set
	\begin{equation*}
		\hat{Z}_{\chi} = \settc{u \in \mathbf{E}_{*}}{I_{c}[u] = \chi
		\text{ and } I'_{c}[u] = 0 \text{ and }
		\norm{u}_{L^{6}_{w}(\Omega_{*})} \leq \bar{C} }
	\end{equation*}
	is compact in $\mathbf{E}_{*}$, using a property of genus, we can
	find a neighborhood $U$ of $\hat{Z}_{\chi}$ which has finite
	genus, say $\gamma(U) = k_{0} < +\infty$. Let $\mathcal{A} =
	I_{c}^{\chi+\epsilon} \cap \mathcal{B}_{\bar{C}}$. As in proving
	Proposition \ref{prop:PSbounded} and Lemma \ref{lem:gradLip},
	since the Palais-Smale property holds in $\mathcal{A}$, when
	$\epsilon$ is small enough we can find a map $\eta$ such that
	\begin{equation*}
		I_{c}[u] < \chi -\epsilon \qquad \text{for all } u \in
		\eta(\mathcal{A} \setminus U).
	\end{equation*} 
	This implies that $\gamma(\mathcal{A} \setminus U) \leq
	\gamma(\eta(\mathcal{A} \setminus U)) = k_{1} < +\infty$. Then
	$\mathcal{A} = (\mathcal{A} \setminus U) \cup (\mathcal{A} \cap
	U)$ gives
	\begin{equation*}
		\gamma(\mathcal{A}) \leq \gamma(\mathcal{A} \setminus U) +
		\gamma(\mathcal{A} \cap U) \leq k_{1} + k_{0} < +\infty,
	\end{equation*}
	which leads to a contradiction.
\end{proof}

\begin{lemma}
	\label{zero}
	If $\{\hat u_k\}$ is the sequence of critical points obtained by
	Proposition \ref{prop:LS}, then
	\begin{align}
		\lim_{k\rightarrow+\infty}\int_{\Omega_{*}}e^x\hat u_k^4\,
		dxdy = 0, \label{zero1}\\
		\lim_{k\rightarrow+\infty}\int_{\Omega_{*}}e^x\hat u_k^2\,
		dxdy = 0, \label{zero2}\\
		\lim_{k\rightarrow+\infty}\int_{\Omega_{*}}e^x(\hat
		u_k)_x^2dxdy = 0, \label{zero3}\\
		\lim_{k\rightarrow+\infty}\int_{\Omega_{*}}e^x\vert\nabla_y
		\hat u_k\vert^2dxdy = 0. \label{zero4}
	\end{align}
\end{lemma}
\begin{proof}
	Since
	\begin{equation*}
		\int_{\Omega_{*}}\left({c^2}(\hat u_k)_x^2 \, dxdy +
		\vert\nabla_y \hat u_k\vert^2-F'(\hat u_k)u_k\right)e^xdxdy =
		\langle I'_c[\hat u_k],\hat u_k \rangle = 0,
	\end{equation*}
	it follows that 
	\begin{multline}
		\label{-1/4}
		0 = \lim_{k\rightarrow+\infty}I_c[\hat u_k] \\
		=\lim_{k\rightarrow+\infty}\int_{\Omega_{*}}
		\left(\frac{c^2}{2}(\hat u_k)_x^2 \, dxdy +
		\frac{1}{2}\vert\nabla_y \hat u_k\vert^2+F(\hat u_k)\right)e^x
		\, dxdy \\
		= \lim_{k\rightarrow+\infty}-\frac{1}{4}
		\int_{\Omega_{*}}e^x\hat u_k^4 \, dxdy.
	\end{multline}
	This together with \eqref{Ho} yields
	\begin{equation*}
		\lim_{k\rightarrow+\infty}\int_{\Omega_{*}}e^x\hat u_k^2 \,
		dxdy = 0.
	\end{equation*}
	Combining with \eqref{-1/4} completes the proof.
\end{proof}

\begin{lemma}
	\label{uk-}
	If $u \in \mathbf{E}_{*}$ is a bounded critical point of $I_c$
	then $u_x \in L^2(\Omega_*)$ and
	\begin{align}
		\label{-decay1}
		\lim_{x\to -\infty} u_x(x,y)=0
	\end{align}
	uniformly in $y$.
\end{lemma}

\begin{proof}
	Multiplying \eqref{te} by $u_x$ and integrating by parts, we get
	\begin{multline}
		\label{eee1}
		c^2\int_{x_n}^{0}\int_{\Omega_y} u_x^2 \, dydx\\
		=-\int_{\Omega_y}\left(\frac{c^2}{2}u_x^2-\frac{1}{2}\vert\nabla_y
		u\vert^2-F(u)\right)dy\bigg|_{x=x_n}^{x=0}\\
		-\int_{\partial\Omega_y}\int_{x_n}^{0}u_x\frac{\partial
		u}{\partial\nu_y}dxd\sigma_y,
	\end{multline}
	where $\nu_y$ is a normal vector to $\partial\Omega_y$ on which
	$d\sigma_y$ is an integral element. The last term of \eqref{eee1}
	vanishes since $u_{x} \equiv 0$ on $\partial \Omega_{y}$ due to
	the boundary conditions and hence
	\begin{multline}
		\label{eq:eee2}
		c^2\int_{x_n}^{0}\int_{\Omega_y} u_x^2 \, dydx \\
		= - \int_{\Omega_y}\left( \frac{1}{2} \abs{\nabla_y u(x_{n}, y)}^2
		+ F(u(x_{n}, y)) \right) \, dy \\
		+ \frac{c^2}{2} \int_{\Omega_y} \left( u_x^2(x_{n}, y) -
		u_x^2(0, y) \right) \, dy.
	\end{multline}
	
	Using the facts that $u$ and $\nabla u$ are uniformly bounded, we
	arrive at
	\begin{equation}
		\label{i0}
		\int_{x_n}^{0}\int_{\Omega_y} u_x^2 \, dxdy\leq C
	\end{equation}
	with $C$ being a constant independent of $n$. Passing to the limit
	as $n \to \infty$ yields $u_x\in L^2(\Omega_*)$. Then
	\eqref{-decay1} follows, since $u_x$ is uniformly continuous in
	$\Omega_*$.
\end{proof}

\begin{lemma}
	\label{uk-1}
	Suppose that $J$ has only isolated critical points in
	$H^1_0(\Omega_y)$ and $u$ is a nonconstant critical point of $I_c$
	obtained by Proposition~\ref{prop:LS}. Then
	\begin{equation}
		\label{-infy}
		\lim_{x \to -\infty}u(x,y) = v(y) \text{ uniformly in $y$}
	\end{equation}
	and $v$ is a critical point of $J$ with $J[v]<0$. Furthermore if
	$\lambda_1 < 1\leq \lambda_2$ then $v=u_+$ or $-u_+$.
\end{lemma}

\begin{proof}
	We first show that for any sequence $x_{n} \to -\infty$ there
	exist a subsequence $x_{n_{k}}$ and a critical point $v(y)$ of $J$
	such that
	\begin{equation*}
		u(x_{n_{k}}, y) \to v(y) \text{ in } C^1(\bar \Omega_0),
	\end{equation*}
	where $\Omega_0=(-1,0)\times\Omega_y$.

	Take any sequence $x_{n} \to - \infty$. By Proposition
	\ref{prop:nobdry}, for all $n \in \bf N$, $\norm{u(x + x_{n},
	y)}_{C^{1, \alpha}(\bar \Omega_0)}$ are uniformly bounded. Hence
	along a subsequence $x_{n_{k}}$
	\begin{equation}
		\label{v}
		u(x + x_{n_{k}} ,y) \to v(x,y) \text{ in } C^1(\bar \Omega_0).
	\end{equation}
	It follows from \eqref{-decay1} that $v_x(x,y) \equiv 0$; thus $v
	\in C^1(\bar \Omega_y)$, a function which depends on $y$ only.
	
	Let $\phi\in H^1_0(\Omega_y)$. Multiplying \eqref{te} by $\phi$
	and integrating over $\Omega_0$, we get
	\begin{multline*}
		c^2\int_{\Omega_y}(u_x(x + x_{n_{k}},y) + u(x + x_{n_{k}}, y))
		\cdot \phi(y) \, dy \big|^{x=1}_{x=0} \\
		-\int_{\Omega_0}[\nabla_y u(x + x_{n_{k}},y) \cdot
		\nabla_y\phi(y) - f(u(x + x_{n_{k}},y))\phi(y)] \, dydx = 0.
	\end{multline*}
	Passing to the limit as $k \to \infty$, we use \eqref{-decay1} and
	\eqref{v} to obtain
	\begin{equation*}
		\int_{\Omega_0}[\nabla_y v(y) \cdot \nabla_y\phi(y) -
		f(v(y))\phi(y)] \, dxdy = 0.
	\end{equation*}
	Then 
	\begin{equation*}
		\int_{\Omega_y}[\nabla_y v \cdot \nabla_y\phi - f(v)\phi] \,
		dy = 0,
	\end{equation*}
	which shows $v$ is a critical point of $J$ and our claim hods.
	Invoking $(\ref{-decay1})$ and letting $k \to \infty$ in
	\eqref{eee1}, we also get
	\begin{multline}
		\label{-l}
		J[v] = \int_{\Omega_y}[\frac12|\nabla_y v|^2+F(v)]dy \\
		=-c^2 \int_{\Omega_*} u_x^2 \, dydx -
		\frac{c^2}{2}\int_{\Omega_y} u_x^2(0,y) \, dy < 0.
	\end{multline}
	From the above equality we deduce that, while $v$ in principle
	depends on the sequence $\{x_{n}\}$ and its subsequence $n_{k}$,
	the critical value $J[v]$ does not.
	
	To show \eqref{-infy}, we claim
	\begin{equation*}
		u(x + x_n,y) \to v(y) \text{ in } C^1(\bar \Omega_y) \text{
		along any sequence $x_n \to-\infty$}.
	\end{equation*}
	For otherwise, there exists a decreasing sequence $x_n \to-\infty$
	such that
	\begin{equation} 
		\label{oddeven}
		u(x+x_n,y) \to \tilde v(y)\mbox{ if n is odd, }u(x+x_n,y) \to
		v(y)\mbox{ if n is even},
	\end{equation}
	and 
	\begin{equation*}
		\label{v=}
		\kappa := \| \tilde v - v \|_{C(\bar \Omega_y)} > 0.
	\end{equation*}
	It follows from \eqref{-l} that
	\begin{equation*}
		E(v) = E(\tilde v).
	\end{equation*}

	(i) Suppose that there exists a decreasing sequence $x_n
	\to-\infty$ such that \eqref{oddeven} holds and $|x_{n+1}- x_n|
	\le M$ for all $n \in \mathbb{N}$. If $n$ is large then there
	exist $y \in \Omega_y$ and
	$\xi_n\in(x_{n+1},x_n)$ such that 
	\begin{equation*}
		|u_x(\xi_n,y)| = \frac{|u(x_n,y) - u(x_{n+1},y)|}{|x_n -
		x_{n+1}|} \ge \frac{\kappa}{3M}.
	\end{equation*}
	This contradicts $(\ref{-decay1})$.

	(ii) It remains to treat the case in which $|x_{n+1}- x_n| \to
	\infty$, $\|u(x + x_n,y) - \tilde v(y)\|_{C^1(\bar \Omega_y)} \to
	0$ and $\|u(x + x_n,y) - v(y)\|_{C^1(\bar \Omega_y)} \to 0$ as $n
	\to \infty$. From Lemma \ref{uk-}, we know $u_x \in
	L^2(\Omega_*)$. Hence there exists a sequence $\{\eta_n\}$ with
	$\lim_{n \to \infty} \eta_n = 0$ such that $\norm{u(x + x_{n}, y)-
	\tilde v(y)}_{C(\bar \Omega_0)} < \eta_n$, $\norm{u(x + x_{n+1},
	y)- v(y)}_{C(\bar \Omega_0)} < \eta_n$ and
	\begin{equation*}
		\label{eta}
		\int_{x_{n}}^{x_{n+1}}\int_{\Omega_y} |u_x(x,y)|^2 \, dydx <
		\eta_n.
	\end{equation*}
	Since $J$ has only isolated critical points in $H^1_0(\Omega_y)$,
	there exist $\kappa_1,\kappa_2 \in (0,\kappa)$ such that $w$ is
	not a critical point of $J$ if $\kappa_1 \leq \norm{w- v}_{C(\bar
	\Omega_y)} \leq \kappa_2$. Since $\norm{u(x + \xi, y)}_{C(\bar
	\Omega_0)}$ is continuous with respect to $\xi$, there exists
	$\bar\xi_n\in(x_{n+1},x_n)$ such that $\norm{u(x + \bar\xi_n, y)-
	v}_{C(\bar \Omega_0)}=\frac{\kappa_1+\kappa_2}{2}$. Arguing like
	in (i), we see that $|\bar\xi_n - x_n| \to \infty$ and $|\bar\xi_n
	- x_{n+1}| \to \infty$ as ${n \to \infty}$. Set
	$V_n(x,y)=u(x+\bar\xi_n, y)$. We then know that, along a
	subsequence, still denoted by $\{V_n\}$, we have that $V_n(x,y)
	\to V(y)$ in $C^1(\bar \Omega_0)$ with $V(y)$ a critical point of
	$J$. This is not possible, since
	\begin{equation*}
		\norm{V(y) - v(y)}_{C(\bar \Omega_y)} = \lim_{n \to - \infty}
		\norm{u(x + \bar\xi_n, y)- v(y)}_{C(\bar \Omega_0)} =
		\frac{\kappa_1+\kappa_2}{2}.
	\end{equation*}

	The last assertion follows from Proposition~\ref{y}. Now the proof
	is complete.

\end{proof}

\section{Passing to limit from approximate solutions} 
\label{sec_lim}

Let $\{\hat u_k\}$ be a sequence of solutions obtained in section
\ref{sec_variation}. First we consider the case that $\lambda_{1} < 1
\leq \lambda_{2}$. Then along a subsequence
\begin{equation}
	\label{-ify+}
	\lim_{x\to -\infty} \hat u_k(x,y) = u_+(y)
\end{equation}
or
\begin{equation*}
	\lim_{x\to -\infty} \hat u_k(x,y) = -u_+(y).
\end{equation*}
We may assume \eqref{-ify+} holds, for otherwise taking $-\hat u_k$
will do.

By \eqref{zero3} and Proposition~\ref{prop:LS}
\begin{align}
	&I_c[\hat u_k] \leq I_c[\hat u_{k+1}] < 0, \label{K1}\\ 
	&\lim_{k\rightarrow+\infty}I_c[\hat u_k] = 0, \label{K2} \\
	&\lim_{k \rightarrow + \infty} \int_{\Omega_*} e^x (\hat
	u_k)_x^2 \, dxdy = 0, \label{K4}
\end{align}
while from \eqref{h4} and \eqref{pineq} we deduce
\begin{align}
	\label{K3}
	0 < \int_{\Omega_*} e^x(\hat u_k)_x^2 \, dxdy\leq \frac{2}{c^2}
	\int_{\Omega_*} e^x \, dxdy, \\
	\label{eq:stimaSob}
	\int_{\Omega_{*}} e^x \hat{u}_{k}^2 \, dxdy \leq 4
	\int_{\Omega_{*}} e^x (\hat{u}_{k})_x^2\, dx dy.
\end{align}
Furthermore using \eqref{<0} yields 
\begin{align}
	\label{eq:KK}
	&\int_{\Omega_*} e^x \hat{u}_k^{4} \, dxdy \leq 2
	\int_{\Omega_*} e^x \hat{u}_k^{2} \, dxdy,\\
	& 
	\label{eq:KKK}
	\int_{\Omega_*} e^x \abs{\nabla_{y} \hat{u}_k}^{2} \, dxdy \leq
	\int_{\Omega_*} e^x \hat{u}_k^{2} \, dxdy.
\end{align}

\begin{proof}[Proof of theorem \ref{main}]
	We prove (i) first. Let $\mu =
	\norm{v_{+}}_{H^{1}_{0}(\Omega_{y})}$ and $x_{k} $ be the largest
	real number $\bar{x}$ satisfying
	\begin{equation}
		\label{eq:stimeLimite}
		\int_{\bar{x}-1}^{\bar{x}} \int_{\Omega_{y}} (\abs{\nabla_{y}
		\hat{u}_{k}(x,y) - \nabla_{y} v_{+}(y)}^{2} +
		\abs{\hat{u}_{k}(x,y) - v_{+}(y)}^{2})dydx = \frac{\mu}{8}
	\end{equation}
	and 
	\begin{equation*}
		{\int_{{z}-1}^{{z}}} \int_{\Omega_{y}} (\abs{\nabla_{y}
		\hat{u}_{k}(x,y) - \nabla_{y} v_{+}(y)}^{2} +
		\abs{\hat{u}_{k}(x,y) - v_{+}(y)}^{2})dydx < \frac{\mu}{8}
	\end{equation*}
	if $z < \bar{x}$. From \eqref{K4}, \eqref{eq:stimaSob},
	\eqref{eq:KK} and \eqref{eq:KKK}, we deduce that for all $z < 0$
	\begin{equation*}
		\int_{{z}-1}^{{z}} \int_{\Omega_{y}} (\abs{\nabla_{y}
		\hat{u}_{k}(x,y)}^{2} + \abs{\hat{u}_{k}(x,y)}^{2})dydx \to 0
	\end{equation*}
	as $k \to +\infty$. This implies $x_{k} \to -\infty$. Define
	\begin{equation} 
		\label{ap1}
		w_k(x,y)= 
		\begin{cases}
			\hat{u}_k (x+x_k,y) & \text{if } x \leq -x_k, \\
			0 & \text{if } x > -x_k.
		\end{cases}
	\end{equation}
	It is clear that $w_{k}(x,y) \to v_{+}(y)$ as $x \to -\infty$ and
	$w_{k}(x,y) \to 0$ as $x \to +\infty$. Along a subsequence
	$w_{k}(x,y) \to U(x,y)$ in $C^{2}_{\text{loc}}$, and $U$ is a
	bounded solution of \eqref{te}. By 
	\eqref{eq:stimeLimite} 
	\begin{equation*}
		\int_{-1}^{0} \int_{\Omega_{y}} (\abs{\nabla_{y} U(x,y) -
		\nabla_{y} v_{+}(y)}^{2} + \abs{U(x,y) - v_{+}(y)}^{2})dydx =
		\frac{\mu}{8}, 
	\end{equation*}
	which ensures that $U(x,y)$ is a nontrivial solution of
	\eqref{te}. We remark that for all $a,b \in \mathbb{R}$ and $a <
	b$,
	\begin{equation*}
		\int_{a}^{b} \int_{\Omega_{y}} U^{2}_{x}(x,y) \, dx dy \leq
		\lim_{k \to +\infty} \int_{a}^{b} \int_{\Omega_{y}}
		(w_{k})^{2}_{x}(x,y) \, dx dy.
	\end{equation*}
	From the proof of \eqref{i0}, we know
	\begin{equation*}
		\int_{-\infty}^{0} \int_{\Omega_{y}} (\hat{u}_{k})_{x}^{2} \,
		dx dy
	\end{equation*}
	is bounded. Hence $U_{x} \in L^{2}(\Omega)$ and $U_{x}(x,y) \to 0$
	as $x \to \pm\infty$. Arguing like Lemma \ref{uk-1}, we deduce
	that
	\begin{equation*}
		v_{-\infty}(y) = \lim_{x \to -\infty} U(x,y), \qquad
		v_{\infty}(y) = \lim_{x \to +\infty} U(x,y),
	\end{equation*}
	where $v_{-\infty}$ and $v_{\infty}$ are the solutions of
	\eqref{b}.
	
	As in the proof of \eqref{eq:eee2}, we have 
	\begin{equation}
		\label{eq:eee3}
		\begin{aligned}
			c^2\int_{a}^{b}\int_{\Omega_y} U_x^2 \, dydx &= -
			\int_{\Omega_y}\left( \frac{1}{2} \abs{\nabla_y U(b, y)}^2
			+ F(U(b, y)) \right) \, dy \\
			&\quad + \int_{\Omega_y}\left( \frac{1}{2} \abs{\nabla_y
			U(a, y)}^2 + F(U(a, y)) \right) \, dy \\
			&\quad + \frac{c^2}{2} \int_{\Omega_y} \left( U_x^2(a, y) -
		U_x^2(b, y) \right) \, dy.
		\end{aligned}
	\end{equation}
	Letting $a \to -  \infty$ and $b \to +\infty$ gives
	\begin{equation} \label{v-+}
		-J[v_{-\infty}] + J[v_{\infty}] =
		c^2\int_{-\infty}^{\infty}\int_{\Omega_y} U_x^2 \, dydy > 0.
	\end{equation}
	This implies $v_{-\infty}=u_+$ and $v_{\infty}=0$, which completes
	the proof of (i).
	
	The proof of (iii) is trivial. It remains to show (ii). Since $J$
	has finite number of critical points in $H^1_0(\Omega_y)$, there
	is a subsequence of $\{\hat u_k\}$, still denoted by $\{\hat
	u_k\}$, such that
	\begin{equation*}
		\lim_{x\to -\infty} \hat u_k(x,y) = u_*(y)
	\end{equation*}
	and $u_*$ is a solution of \eqref{b}. With a slight modification,
	we obtain a bounded nontrivial solution $U(x,y)$ of \eqref{te} and
	\eqref{v-+} holds. Clearly $v_{-\infty}=u_*$. Then setting
	$v_{\infty}=u^*$ completes the proof.
        
\end{proof}

\nocite{Allen_Cahn_1979, Berestycki_Hamel_Nadirashvili_2010}
 
\providecommand{\bysame}{\leavevmode\hbox to3em{\hrulefill}\thinspace}
\providecommand{\MR}{\relax\ifhmode\unskip\space\fi MR }
\providecommand{\MRhref}[2]{%
  \href{http://www.ams.org/mathscinet-getitem?mr=#1}{#2}
}
\providecommand{\href}[2]{#2}


\begin{thebibliography}{10}

\bibitem{Allen_Cahn_1979}
S.~M. Allen and J.~W. Cahn, \emph{A microscopic theory for antiphase
boundary motion and its application to antiphase domain coarsening},
Acta Metall. \textbf{27} (1979), 1085--1095.

\bibitem{Ambrosetti_Malchiodi_2007}
Antonio Ambrosetti and Andrea Malchiodi, \emph{Nonlinear analysis and
semilinear elliptic problems}, Cambridge Studies in Advanced
Mathematics, vol. 104, Cambridge University Press, Cambridge, 2007.
\MR{2292344 (2008k:35129)}

\bibitem{Aronson_Weinberger_1975}
D.~G. Aronson and H.~F. Weinberger, \emph{Nonlinear diffusion in
population genetics, combustion, and nerve pulse propagation}, Partial
differential equations and related topics ({P}rogram, {T}ulane
{U}niv., {N}ew {O}rleans, {L}a., 1974), 1975, pp.~5--49. Lecture Notes
in Math., Vol. 446. \MR{0427837}
  
\bibitem{BTW} 
{J. M. Ball, A Taheri and M. Winter, Local minimizers in
micromagnetics and related problems. Calc. Var. Partial Differential
Equations 14 (2002), no. 1, 1-27.}

\bibitem{BLL} 
H. Berestycki, B. Larrouturou and P.-L. Lions, Multi-dimensional
travelling wave solutions of a flame propagation model. Arch. Rational
Mech. Anal. 111 (1990), no. 1, 33-49.

\bibitem{Berestycki_Hamel_Nadirashvili_2010}
Henri Berestycki, Fran\c{c}ois Hamel, and Nikolai Nadirashvili,
\emph{The speed of propagation for {KPP} type problems. {II}.
{G}eneral domains}, J. Amer. Math. Soc. \textbf{23} (2010), no.~1,
1--34. \MR{2552247}
  
  \bibitem{Berestycki_Nirenberg_1991}
H.~Berestycki and L.~Nirenberg, \emph{On the method of moving planes
and the sliding method}, Bol. Soc. Brasil. Mat. (N.S.) \textbf{22}
(1991), no.~1, 1--37. \MR{1159383}

\bibitem{Berestycki_Nirenberg_1992}
Henri Berestycki and Louis Nirenberg, \emph{Travelling fronts in
cylinders}, Ann. Inst. H. Poincar\'{e} Anal. Non Lin\'{e}aire
\textbf{9} (1992), no.~5, 497--572. \MR{1191008}

\bibitem{Chen_1990}
Chao-Nien Chen, \emph{Multiple solutions for a class of nonlinear
{S}turm-{L}iouville problems on the half line}, J. Differential
Equations \textbf{85} (1990), no.~2, 236--275. \MR{1054550}

\bibitem{Chen_1990-1}
\bysame, \emph{Uniqueness and bifurcation for solutions of nonlinear
{S}turm-{L}iouville eigenvalue problems}, Arch. Rational Mech. Anal.
\textbf{111} (1990), no.~1, 51--85. \MR{1051479}

\bibitem{Chen_1993}
\bysame, \emph{Uniqueness of solutions of some second order
differential equations}, Differential Integral Equations \textbf{6}
(1993), no.~4, 825--834. \MR{1222303}
 
 \bibitem{dKW} 
{Manuel del Pino, Michal Kowalczyk and Juncheng Wei, Entire solutions
of the Allen-Cahn equation and complete embedded minimal surfaces of
finite total curvature in $\mathbb{R}^3$, J. Differential Geom. 93,
(2013), 67-131.}

\bibitem{Fife_Mcleod_1977}
Paul~C. Fife and J.~B. McLeod, \emph{The approach of solutions of
nonlinear diffusion equations to travelling front solutions}, Arch.
Ration. Mech. Anal. \textbf{65} (1977), no.~4, 335--361. \MR{0442480}

\bibitem{Gallay_Risler_2007}
Thierry Gallay and Emmanuel Risler, \emph{A variational proof of
global stability for bistable travelling waves}, Differential Integral
Equations \textbf{20} (2007), no.~8, 901--926. \MR{2339843}

\bibitem{Gardner_1986}
Robert Gardner, \emph{Existence of multidimensional travelling wave
solutions of an initial-boundary value problem}, J. Differential
Equations \textbf{61} (1986), no.~3, 335--379. \MR{829368}

\bibitem{Gilbarg_Trudinger_1983}
David Gilbarg and Neil~S. Trudinger, \emph{Elliptic partial
differential equations of second order}, second ed., Grundlehren der
Mathematischen Wissenschaften [Fundamental Principles of Mathematical
Sciences], vol. 224, Springer-Verlag, Berlin, 1983. \MR{737190}

\bibitem{Heinz_1986}
Hans-Peter Heinz, \emph{Nodal properties and bifurcation from the
essential spectrum for a class of nonlinear {S}turm-{L}iouville
problems}, J. Differential Equations \textbf{64} (1986), no.~1,
79--108. \MR{849666}

\bibitem{Heinz_1986-1}
\bysame, \emph{Nodal properties and variational characterizations of
solutions to nonlinear {S}turm-{L}iouville problems}, J. Differential
Equations \textbf{62} (1986), no.~3, 299--333. \MR{837759}

\bibitem{Heinz_1987}
\bysame, \emph{Free {L}justernik-{S}chnirelman theory and the
bifurcation diagrams of certain singular nonlinear problems}, J.
Differential Equations \textbf{66} (1987), no.~2, 263--300.
\MR{871998}

\bibitem{Heinze_2001}
S.~Heinze, \emph{A variational approach to traveling waves}, Tech.
Report~85, Max Planck Institute for Mathematical Sciences, 2001.

\bibitem{Lucia_Muratov_Novaga_2008}
M.~Lucia, C.~B. Muratov, and M.~Novaga, \emph{Existence of traveling
waves of invasion for {G}inzburg-{L}andau-type problems in infinite
cylinders}, Arch. Ration. Mech. Anal. \textbf{188} (2008), no.~3,
475--508. \MR{2393438}

\bibitem{Lucia_Muratov_Novaga_2004}
Marcello Lucia, Cyrill~B. Muratov, and Matteo Novaga, \emph{Linear vs.
nonlinear selection for the propagation speed of the solutions of
scalar reaction-diffusion equations invading an unstable equilibrium},
Comm. Pure Appl. Math. \textbf{57} (2004), no.~5, 616--636.
\MR{2032915}
  
\bibitem{M} 
L. Modica, The gradient theory of phase transitions and the minimal
interface criterion. Arch. Rat. Mech. Anal. 98 (1987),123-142.


\bibitem{Muratov_2004}
C.~B. Muratov, \emph{A global variational structure and propagation of
disturbances in reaction-diffusion systems of gradient type}, Discrete
Contin. Dyn. Syst. Ser. B \textbf{4} (2004), no.~4, 867--892.
\MR{2082914}

\bibitem{Rabinowitz_1971}
Paul~H. Rabinowitz, \emph{A note on a nonlinear eigenvalue problem for
a class of differential equations}, J. Differential Equations
\textbf{9} (1971), 536--548. \MR{0273099}

\bibitem{Rabinowitz86-1}
\bysame, \emph{Minimax methods in critical point theory with
applications to differential equations}, CBMS Regional Conference
Series in Mathematics, vol.~65, Published for the Conference Board of
the Mathematical Sciences, Washington, {DC}, 1986. \MR{MR845785
(87j:58024)}
  
\bibitem{RS} 
{Paul~H. Rabinowitz and Ed Stredulinsky, Mixed states for an
Allen-Cahn type equation. Dedicated to the memory of J\"urgen K.
Moser. Comm. Pure Appl. Math. 56 (2003), no. 8, 1078-1134.}

\bibitem{Reineck_1988}
James~F. Reineck, \emph{Travelling wave solutions to a gradient
system}, Trans. Amer. Math. Soc. \textbf{307} (1988), no.~2, 535--544.
\MR{940216}

\bibitem{Rinzel_Terman_1982}
John Rinzel and David Terman, \emph{Propagation phenomena in a
bistable reaction-diffusion system}, {SIAM} J. Appl. Math. \textbf{42}
(1982), no.~5, 1111--1137. \MR{673529}

\bibitem{Risler_2008}
Emmanuel Risler, \emph{Global convergence toward traveling fronts in
nonlinear parabolic systems with a gradient structure}, Ann. Inst. H.
Poincar\'{e} Anal. Non Lin\'{e}aire \textbf{25} (2008), no.~2,
381--424. \MR{2400108}
  
\bibitem{S}  
A. Scheel, Coarsening fronts. Arch. Ration. Mech. Anal. 181 (2006),
no. 3, 505-534.

\bibitem{Vega_1993}
Jos\'{e}~M. Vega, \emph{The asymptotic behavior of the solutions of
some semilinear elliptic equations in cylindrical domains}, J.
Differential Equations \textbf{102} (1993), no.~1, 119--152.
\MR{1209980}

\bibitem{Vega_1993-1}
\bysame, \emph{Travelling wavefronts of reaction-diffusion equations
in cylindrical domains}, Comm. Partial Differential Equations
\textbf{18} (1993), no.~3-4, 505--531. \MR{1214870}

\bibitem{Volpert_Volpert_Volpert_1994}
Aizik~I. Volpert, Vitaly~A. Volpert, and Vladimir~A. Volpert,
\emph{Traveling wave solutions of parabolic systems}, Translations of
Mathematical Monographs, vol. 140, American Mathematical Society,
Providence, {RI}, 1994, Translated from the Russian manuscript by
James F. Heyda. \MR{1297766}

\bibitem{Willem_96}
Michel Willem, \emph{Minimax theorems}, Progress in Nonlinear
Differential Equations and their Applications, 24, Birkh\"auser Boston
Inc., Boston, {MA}, 1996. \MR{MR1400007 (97h:58037)}

\end{thebibliography}
\end{document}